\documentclass[11pt,reqno]{amsart}

\usepackage{color}
\usepackage[colorlinks=true, allcolors=blue]{hyperref}
\usepackage{amsmath, amssymb, amsthm}
\usepackage{mathrsfs}
\usepackage{mathtools}
\usepackage[noabbrev,capitalize,nameinlink]{cleveref}
\crefname{equation}{}{}
\usepackage[margin=1in]{geometry}
\usepackage{graphics}
\usepackage{pifont}
\usepackage{tikz}
\usepackage{bm,bbm}
\usepackage[T1]{fontenc}

\usetikzlibrary{arrows.meta}

\usepackage{environ}
\usepackage{framed}
\usepackage{url}
\usepackage{graphicx}
\usepackage[textsize=tiny]{todonotes}
\usepackage{enumerate}
\allowdisplaybreaks[1]
\usepackage{enumerate}

\crefname{equation}{}{} 

\numberwithin{equation}{section}
\newtheorem{theorem}{Theorem}[section]
\newtheorem{proposition}[theorem]{Proposition}
\newtheorem{lemma}[theorem]{Lemma}

\crefname{claim}{Claim}{Claims}

\newtheorem{conjecture}[theorem]{Conjecture}
\newtheorem*{question*}{Question}

\theoremstyle{definition}
\newtheorem{definition}[theorem]{Definition}

\newtheorem*{definition*}{Definition}

\theoremstyle{remark}
\newtheorem*{remark}{Remark}

\newcommand{\snorm}[1]{\lVert#1\rVert}
\newcommand{\sang}[1]{\langle #1 \rangle}
\newcommand{\sabs}[1]{\lvert#1\rvert}

\newcommand{\mb}{\mathbb}
\newcommand{\mbf}{\mathbf}
\newcommand{\mbm}{\mathbbm}
\newcommand{\mc}{\mathcal}
\newcommand{\mf}{\mathfrak}

\newcommand{\msf}{\mathsf}

\newcommand{\on}{\operatorname}

\DeclareMathOperator{\proj}{proj}
\DeclareMathOperator{\Gr}{Gr}

\title{The cylindrical width of transitive sets}

\author[Sah]{Ashwin Sah}
\author[Sawhney]{Mehtaab Sawhney}
\author[Zhao]{Yufei Zhao}

\thanks{Sah and Sawhney were supported by NSF Graduate Research Fellowship Program DGE-1745302. Zhao was supported by NSF Award DMS-1764176, a Sloan Research Fellowship, and the MIT Solomon Buchsbaum Fund.}

\address{Department of Mathematics, Massachusetts Institute of Technology, Cambridge, MA 02139, USA}
\email{\{asah,msawhney,yufeiz\}@mit.edu}

\begin{document}

\begin{abstract}
We show that for every $1 \le k \le d/(\log d)^C$, every finite transitive set of unit vectors in $\mb{R}^d$ lies within distance $O(1/\sqrt{\log (d/k)})$ of some codimension $k$ subspace, and this distance bound is best possible. 
This extends a result of Ben Green, who proved it for $k=1$.
\end{abstract}

\maketitle

\section{Introduction}\label{sec:introduction}

The following counterintuitive fact was conjectured by the third author and proved by Green~\cite{Gre20}. It says that every finite transitive subset of a high dimensional sphere is close to some hyperplane. 
Here a subset $X$ of a sphere in $\mb{R}^d$ is \emph{transitive} if for every $x,x' \in X$, there is some $g \in \msf{O}(\mb{R}^d)$ so that $g X = X$ and $gx = x'$. We say that $X$ has \emph{width} at most $2r$ if it lies within distance $r$ of some hyperplane. The finiteness assumption is important since otherwise the whole sphere is a counterexample. 

\begin{theorem}[Green~\cite{Gre20}]\label{thm:green}
Let $X$ be a finite transitive subset of the unit sphere in $\mb{R}^d$.
Then the width of $X$ is at most $O(1/\sqrt{\log d})$.
Furthermore, this upper bound is best possible up to a constant factor.
\end{theorem}

The bound in the theorem is tight since the set $X$ obtained by taking all permutations and coordinate-wise $\pm$ signings of the unit vector $(1, 1/\sqrt{2}, \ldots, 1/\sqrt{d})/\sqrt{H_d}$, where $H_d = 1 + 1/2 + \cdots + 1/d \sim \log d$, has width on the order of $1/\sqrt{\log d}$.

Green's proof uses a clever induction scheme along with sophisticated group theoretic arguments, including an application of the classification of finite simple groups.

We generalize Green's result by showing that a finite transitive set lies not only near some hyperplane, but in fact it lies near a subspace of codimension $k$, as long as $k$ is not too large.

We say that $X \subset \mb{R}^d$ has \emph{$k$-cylindrical width} at most $2r$ if $X$ lies within distance $r$ of some affine codimension $k$ subspace. The case $k=1$ corresponds to the usual notion of width. 
Our main result below implies that every finite transitive subset of the unit sphere in $\mb{R}^d$ has $k$-cylindrical width $O(1/\sqrt{\log(d/k)})$ as long as $k$ is not too large.

\begin{theorem}\label{thm:real}
There is an absolute constant $C > 0$ so that the following holds.
Let $1 \le k \le d/(\log (3d))^C$.
Let $X$ be a finite transitive subset of the unit sphere in $\mb{R}^d$.
Then there is a real $k$-dimensional subspace $W$ such that
\[
\sup_{\mbf{x}\in X}\snorm{\proj_W\mbf{x}}_2
\lesssim 
\frac{1}{\sqrt{\log(d/k)}}.
\]
\end{theorem}

Here and throughout $a \lesssim b$ means that $a \le C' b$ for some absolute constant $C'$. We write $\snorm{\mbf{x}}_2$ for the usual Euclidean norm of a vector $\mbf{x}$. 
Also $\proj_W$ is the orthogonal projection onto $W$.

We deduce the above theorem from a complex version using a theorem on restricted invertibility  (see \cref{sec:real}).
A transitive subset of the complex unit sphere is defined to be the orbit of a point under the action of some subgroup of the unitary group.

\begin{theorem}\label{thm:complex}
There is an absolute constant $C > 0$ so that the following holds.
Let $1 \le k \le d/(\log (3d))^C$.
Let $X$ be a finite transitive subset of the unit sphere in $\mb{C}^d$.
Then there is a complex $k$-dimensional subspace $W$ such that
\[
\sup_{\mbf{x}\in X}\snorm{\proj_W\mbf{x}}_2
\lesssim 
\frac{1}{\sqrt{\log(d/k)}}.
\]
\end{theorem}

We suspect that the $1 \le k \le d/(\log (3d))^C$ hypothesis is unnecessary in both \cref{thm:real,thm:complex}.



\begin{conjecture}\label{conj:thin-subspace}
Let $1\le k\le d$. Let $X$ be a finite transitive subset of the unit sphere in $\mb{C}^d$.
Then there is a complex $k$-dimensional subspace $W$ such that
\[
\sup_{\mbf{x}\in X}\snorm{\proj_W\mbf{x}}_2
\lesssim 
\frac{1}{\sqrt{\log(2d/k)}}.
\]
\end{conjecture}

One particularly intriguing special case of \cref{conj:thin-subspace} is that every finite transitive set of unit vectors in $\mb{R}^d$ has $k$-cylindrical width $o(1)$ for all $k=o(d)$.

We prove a matching lower bound on the cylindrical radius (See \cref{sec:optimality} for proof.)

\begin{theorem}\label{thm:lower-bound}
Let $1\le k\le d$. There exists a transitive set $X$ in $\mb{R}^d$ such that for any (real or complex) $k$-dimensional subspace $W$ we have
\[\sup_{\mbf{x}\in X}\snorm{\proj_W\mbf{x}}_2
\gtrsim \frac{1}{\sqrt{\log(2d/k)}}.\]
\end{theorem}


We propose another closely related conjecture: every finite transitive set in $\mb{R}^d$ lies inside a small cube.

\begin{conjecture}\label{conj:thin-basis}
Let $X$ be a finite transitive subset of the unit sphere in $\mb{R}^d$ (or $\mb{C}^d$).
Then there is a unitary basis $L$ such that
\begin{equation}\label{eq:thin-basis}
\sup_{\mbf{x}\in X,\mbf{v}\in L}|\sang{\mbf{v},\mbf{x}}|\lesssim \frac{1}{\sqrt{\log d}}.
\end{equation}
\end{conjecture}
Establishing an upper bound that decays to zero as $d \to \infty$ would already be interesting. 
Note that \cref{thm:complex} implies the existence of a set $L$ of orthonormal vectors with $\sabs{L} \ge d^{0.99}$ so that \cref{eq:thin-basis} holds (and likely extendable to $\sabs{L} \ge d/(\log d)^C$ via our techniques). 
Proving either conjecture in full appears to require additional ideas.

\begin{remark}
Green's proof \cite{Gre20} of \cref{thm:real,thm:complex} in the case $k = 1$ contains two errors.
The first error is due to a missing supremum inside the integral in the first and second lines of the last display equation in proof of Proposition 2.1 on page 560.
The second error occurs at the final equality step of the top display equation on page 569, after right after (4.4); here an orthogonality relation was incorrectly applied as it requires an unjustified exchange of the integral and supremum. Our proof here corrects these errors. Green has also updated the arXiv version of his paper \cite{Gre20} incorporating these corrections.
\end{remark}

\section{Proof strategy}\label{sec:preliminaries}

The subspace $W$ in \cref{thm:complex} must vary according to the transitive set $X$. On other hand, the strategy is to construct a single probability distribution $\mu$ (depending only on the symmetry group $G \leqslant \msf{U}(\mb{C}^d)$ but not on $X$) on the set $\Gr_\mb{C}(k, d)$ of $k$-dimensional subspaces of $\mb{C}^d$. This is an important idea introduced by Green (for $k=1$).

\begin{definition}\label{def:k-dimensional-f}
Let $1 \le k \le d$.
Let $f_k(d)$ be the smallest value so that for every finite $G\leqslant\msf{U}(\mb{C}^d)$, there is a probability measure $\mu$ on $\on{Gr}_{\mb{C}}(k,d)$ such that for all $\mbf{v}\in\mb{S}(\mb{C}^d)$,
\[
\int\sup_{g\in G}\snorm{\proj_W(g\mbf{v})}_2^2d\mu(W)\le f_k(d)^2
\]
\end{definition}

The values $f_k(d)$ are well defined since the space of probability measures $\mu$ in question is closed under weak limits.

Our main result about $f_k(d)$ is stated below.

\begin{theorem}\label{thm:f}
If $k\le d/(\log d)^{20}$, then 
\[
f_k(d)\lesssim \frac{1}{\sqrt{\log(d/k)}}.
\]
\end{theorem}

\begin{proof}[Proof of \cref{thm:complex} given \cref{thm:f}]
Let our transitive set $X$ be the orbit of $\mbf{v}\in\mb{S}(\mb{C}^d)$ under the action of the the finite subgroup $G \leqslant \msf{U}(\mb{C}^d)$.
By \cref{thm:f} and \cref{def:k-dimensional-f}, there is a measure $\mu$ on $\on{Gr}_{\mb{C}}(k,d)$ such that
\[\int\sup_{g\in G}\snorm{\proj_W(g\mbf{v})}_2^2d\mu(W)\le f_k(d)^2.\]
Therefore there is some $k$-dimensional subspace $W$ with
\[\sup_{g\in G}\snorm{\proj_W(g\mbf{v})}_2 \le f_k(d) \lesssim\frac{1}{\sqrt{\log(d/k)}}.\qedhere\]
\end{proof}

To prove \cref{thm:f} , we will decompose $G$ to ``smaller'', more restricted cases, namely irreducible and primitive representations. We will also need to consider permutation groups (for both the reduction step as well as the primitive case).

\subsection{Preliminaries}\label{sub:preliminaries}
\begin{definition}\label{def:primitive}
We say that $G\leqslant\msf{U}(\mb{C}^d)$ is \emph{imprimitive} if there is a \emph{system of imprimitivity}: a decomposition
\[\mb{C}^d = \bigoplus_{i=1}^\ell V_i\]
with $0 < \dim V_i < d$ such that for every $g \in G$ and $i\in[\ell]$ one has $gV_i = V_j$ for some $j \in [\ell]$. (The subspaces $V_i$ need not be orthogonal.) 
Otherwise we say that $G$ is \emph{primitive}.
\end{definition}

\begin{remark}
Both primitivity and irreducibility are properties of a representation, rather than intrinsic to a group. We identify $G \leqslant \msf{U}(\mb{C}^d)$ with its natural representation on $\mb{C}^d$.
\end{remark}

It follows from Maschke's theorem that primitive group representations are irreducible. 

\begin{definition}\label{def:decreasing-form}
Given $\mbf{v} = (v_1,\ldots, v_d) \in\mb{C}^d$, let
\[
\mbf{v}^\succ = (\sabs{v_{\sigma(1)}}, \ldots, \sabs{v_{\sigma(d)}}) \in \mb{R}^d
\]
where $\sigma$ is a permutation of $[d]$ so that 
\[
\sabs{v_{\sigma(1)}} \ge  \cdots \ge \sabs{v_{\sigma(d)}}.
\]
We write $v_i^\succ$ for the $i$-th coordinate of $\mbf{v}^\succ$. Let
\[
\on{Dom}(\mbf{v}) = \{ \mbf{w}\in\mb{C}^d : w_i^\succ \le v_i^\succ \text{for all }i \in [d]\}.
\]
\end{definition}

Let (here $\mf{S}_d$ denotes the symmetric group)
\[
\Gamma_d := \mf{S}_d\ltimes(\mb{S}^1)^d\le\msf{U}(\mb{C}^d)
\]
be the group that acts on $\mb{C}^d$ be permuting its coordinates and multiplying individual coordinates by unit complex numbers. Then $\on{Dom}(\mbf{v})$ is the convex hull of the $\Gamma_d$-orbit of $\mbf{v}$.

We define some variants of $f_k(d)$ when the group $G$ is restricted to special types.

\begin{definition}\label{def:k-dimensional-f-2}
Given $k\in[d]$, let $f^{\on{irred}}_k(d)$ (resp. $f^{\on{prim}}_k(d)$) be the smallest value so that for every finite $G\leqslant\msf{U}(\mb{C}^d)$ which is irreducible (resp. primitive), there is a probability measure $\mu$ on $\on{Gr}_{\mb{C}}(k,d)$ such that for every $\mbf{v}\in\mb{S}(\mb{C}^d)$,
\[\int\sup_{g\in G}\snorm{\proj_W(g\mbf{v})}_2^2d\mu(W)\le f^{\on{irred}}_k(d)^2 \qquad \text{(resp. $f^{\on{prim}}_k(d)^2$)}.
\]
\end{definition}

The permutation action on $\mb{C}^d$ deserves special attention. 

\begin{definition}\label{def:alternative-f}
Let $f^{\on{sym}}_k(d)$ be the smallest value so that there is a probability measure $\mu$ on $\on{Gr}_{\mb{C}}(k,d)$ such that for every $\mbf{v} \in \mb{S}(\mb{C}^d)$,
\[
\int\sup_{\mbf{u}\in\on{Dom}(\mbf{v})}\snorm{\proj_W(\mbf{u})}_2^2d\mu(W)\le f^{\on{sym}}_k(d)^2.
\]
Define $f^{\on{alt}}_k(d)$ to be the same with the additional constraint that $\mu$ is supported on the set of $k$-dimensional subspaces of the hyperplane $x_1 + \cdots + x_d = 0$.
\end{definition}

We will often equivalently consider, instead of $\mu$ on $\on{Gr}_{\mb{C}}(k,d)$, the corresponding measure $\mu^\ast$ on the complex Stiefel manifold $V_k(\mb{C}^d)$, that is, $\mu^\ast$ is derived from $\mu$ by first sampling a $\mu$-random $k$-dimensional subspace $W$ of $\mbf{C}^d$, and then outputting a uniformly sampled a unitary basis $(\mbf{w}_1, \ldots, \mbf{w}_k)$ of $W$.
We have 
$\snorm{\proj_W(\mbf{u})}_2^2 = \sum_{\ell=1}^k \sabs{\sang{g \mbf{v}_1, \mbf{w}_\ell}}^2$.

\subsection{Reductions}\label{sub:strategy}
We first reduce the general problem to the irreducible case. 

\begin{proposition}\label{prop:reduction-irred}
If $1\le k < \ell\le d$ then 
\[
f_k(d)\le\max\bigl\{\sqrt{k/\ell},\sup_{d'\ge d/(2\ell)}f^{\on{irred}}_{\lceil 2kd'/d\rceil}(d')\bigr\}.
\]
\end{proposition}

We then reduce the irreducible case to the primitive case and the alternating case.

\begin{proposition}\label{prop:reduction-prim}
If $k\le d/2$, then 
\[
f^{\on{irred}}_k(d)\le\max_{d_1d_2=d}\bigl(\min\bigl\{f^{\on{prim}}_{\lceil k/d_1\rceil}(d_2),f^{\on{alt}}_k(d_1)+\mbm{1}_{k\ge d_1}\bigr\}\bigr).
\]
\end{proposition}

The symmetric and alernating cases can be handled explicitly, yielding the following.

\begin{proposition}\label{prop:permutation}
If $k\le d/(\log d)^5$, then 
\[
f_k^{\on{sym}}(d)\le f_k^{\on{alt}}(d)\lesssim 1/\sqrt{\log(d/k)}.
\]
\end{proposition}

This leaves the primitive case, which we prove by invoking an group theoretic result proved by Green \cite[Proposition~4.2]{Gre20} that allows us to once again reduce to the alternating case once again.

\begin{proposition}\label{prop:f-primitive}
There is an absolute constant $c > 0$ such that for $k\le cd/(\log d)^4$ we have 
\[
f_k^{\on{prim}}(d)\lesssim\sup_{d'\ge cd/(\log d)^4}f^{\on{alt}}_k(d').
\]
\end{proposition}



\subsection{Putting everything together}\label{sub:putting-together}
We are now in position to derive \cref{thm:f} using the preceding statements.

\begin{proposition}\label{prop:primitive}
If $k\le 2d/(\log d)^{10}$ then $f^{\on{prim}}_k(d)\lesssim 1/\sqrt{\log(d/k)}$.
\end{proposition}
\begin{proof}
Combine \cref{prop:permutation,prop:f-primitive}.
\end{proof}
\begin{proposition}\label{prop:irreducible}
If $k\le d/(\log d)^{10}$ then $f^{\on{irred}}_k(d)\lesssim 1/\sqrt{\log(d/k)}$.
\end{proposition}
\begin{proof}
By \cref{prop:reduction-prim}, we have
\[f^{\on{irred}}_k(d)\le\max_{d_1d_2=d}(\min(f^{\on{prim}}_{\lceil k/d_1\rceil}(d_2),f^{\on{alt}}_k(d_1)+\mbm{1}_{k\ge d_1})).\]

First consider the case $d_1\le k$. We have
\[\lceil k/d_1\rceil\le\frac{2d}{d_1(\log d)^{10}}\le\frac{2d_2}{(\log d_2)^{10}}.\]
By \cref{prop:primitive}, we have
\[
f^{\on{prim}}_{\lceil k/d_1\rceil}(d_2) \lesssim \frac{1}{\sqrt{\log(d_2/\lceil k/d_1 \rceil )}} \le \frac{1}{\sqrt{\log(d/(2k))}}.
\]

Now consider the case $d_1 > k$. 
Since $d_2(d_1/k) = d/k$, we have $\max\{d_1, d_2/k\} \ge \sqrt{d/k}$.
If $d_2\ge\sqrt{d/k}$, then
\[
f^{\on{prim}}_{\lceil k/d_1\rceil} (d_2)=f^{\on{prim}}_1(d_2)\lesssim\frac{1}{\sqrt{\log d_2}}\lesssim\frac{1}{\sqrt{\log(d/k)}}.\]
On the other hand, if $d_1/k\ge\sqrt{d/k}$, then $d_1/k\ge(\log d)^5$ so
\[k\le\frac{d_1}{(\log d)^5}\le\frac{d_1}{(\log d_1)^5}.\]
Hence \cref{prop:permutation} yields
\[f^{\on{alt}}_k(d_1)\lesssim\frac{1}{\sqrt{\log(d_1/k)}}\lesssim\frac{1}{\sqrt{\log(d/k)}}.\]

Thus it follows that, for all $d_1d_2 = d$,
\[
\min(f^{\on{prim}}_{\lceil k/d_1\rceil}(d_2),f^{\on{alt}}_k(d_1)+\mbm{1}_{k\ge d_1}) \lesssim \frac{1}{\sqrt{\log(d/k)}},
\]
and the result follows.
\end{proof}

Now we show the main result assuming the above statements.

\begin{proof}[Proof of \cref{thm:f}]
Let $\ell = \lceil\sqrt{dk}\rceil\ge 2k$. We have
\[k/\ell\lesssim\sqrt{k/d}\lesssim\frac{1}{\sqrt{\log(d/k)}}.\]
Also, if $d'\ge d/(2\ell)$ then
\[\bigg\lceil\frac{2kd'}{d}\bigg\rceil\le\frac{d'}{d/(2\ell)}\le\frac{d'}{(\log d)^{10}}\le\frac{d'}{(\log d')^{10}}.\]
By \cref{prop:irreducible}, we have
\[f^{\on{irred}}_{\lceil 2kd'/d\rceil}(d')\lesssim\frac{1}{\sqrt{\log(d'/\lceil 2kd'/d\rceil)}}\lesssim\frac{1}{\sqrt{\log(d/(2\ell))}}\lesssim\frac{1}{\sqrt{\log(d/k)}}.\]
Applying \cref{prop:reduction-irred} to $k$ and $\ell = \lceil\sqrt{dk}\rceil$, we find
\[f_k(d)\le\max(\sqrt{k/\ell},\sup_{d'\ge d/(2\ell)}f^{\on{irred}}_{\lceil 2kd'/d\rceil}(d'))\lesssim\frac{1}{\sqrt{\log(d/k)}}.\qedhere\]
\end{proof}

\subsection{Paper outline}\label{sub:structure}
In \cref{sec:reduction}, we prove the two key reductions, \cref{prop:reduction-irred,prop:reduction-prim}. In \cref{sec:permutation}, we prove the key estimate for the symmetric and alternating cases, \cref{prop:permutation}. In \cref{sec:primitive}, we prove the primitive case, \cref{prop:f-primitive}. Finally, in \cref{sec:real} we deduce a real version from the complex version, proving \cref{thm:real}.
In \cref{sec:optimality} we demonstrate optimality of our results by exhibiting the matching lower bound \cref{thm:lower-bound}.

\section{Reduction to primitive representations}\label{sec:reduction}
We first reduce the general case to the alternating and irreducible cases.

\begin{proof}[Proof of \cref{prop:reduction-irred}]
Consider $G\leqslant\msf{U}(\mb{C}^d)$. By Maschke's theorem, we can decompose into irreducible representations of $G$:
\[\mb{C}^d = \bigoplus_{j=1}^m V_j.\]
Let $d_j = \dim V_j$. Let 
\[J = \{j\in[m]: d_j\ge d/(2\ell)\}.\]

First suppose $\sum_{j \in J} d_j \ge d/2$. Then in each such $V_j$, we consider the probability measure $\mu_j$ that witnesses $f^{\on{irred}}_{\lceil 2kd_j/d\rceil}(d_j)$ for the irreducible representation of $G$ on $V_j$. That is, $\mu_j$ samples a $\lceil 2kd_j/d\rceil$-dimensional subspace of $V_j$ and satisfies
\[
\int\sup_{g\in G}\snorm{\proj_W(g\mbf{v})}_2^2d\mu_j(W)\le f^{\on{irred}}_{\lceil 2kd'/d\rceil}(d_j)^2\snorm{\mbf{v}}_2^2
\]
for each $\mbf{v}\in V_j$. We define $\mu$ to be a uniformly random $k$-dimensional subspace of $\bigoplus_{j \in J} W_j$, where each $W_j$ is an independent $\mu_j$-random $\lceil 2kd_j/d\rceil$-dimensional subspace of $V_j$. (Note the $W_j$'s are orthogonal as the $V_j$'s are.) The total dimension of this direct sum is at least $k$, so $\mu$ is well-defined.

Given $\mbf{v}\in\mb{C}^d$, write $\mbf{v} = \sum_{j=1}^m\mbf{v}_j$ with $\mbf{v}_j\in V_j$.  We have
\begin{align*}
\int\sup_{g\in G}\snorm{\proj_W(g\mbf{v})}_2^2d\mu(W) &\le\int\sup_{g\in G}\snorm{\proj_{\bigoplus_{j\in J}W_j}(g\mbf{v})}_2^2 \prod_{j\in J}d\mu_j(W_j)\\
&\le\sum_{j\in J}\int\sup_{g\in G}\snorm{\proj_{W_j}(g\mbf{v})}_2^2d\mu_j(W_j)\\
&\le\sum_{j\in J}f^{\on{irred}}_{\lceil 2kd_j/d\rceil}(d_j)^2\snorm{\mbf{v}_j}_2^2
\\
&\le\sup_{d'\ge d/(2\ell)}f^{\on{irred}}_{\lceil 2kd'/d\rceil}(d')^2\snorm{\mbf{v}}_2^2
\end{align*}
by orthogonality of the $V_j$.

Next suppose $\sum_{j \in J} d_j < d/2$. Then $|[m]\setminus J|\ge\ell$. Let $I$ be an $\ell$-element subset of $[m]\setminus J$. Choose arbitrary $\mbf{w}_j\in\mb{S}(V_j)\subseteq\mb{C}^d$ for $j\in I$, which are clearly orthogonal. Let $\mu$ be the probability measure on $k$-dimensional subspaces of $\mb{C}^d$ obtained by taking the span of $k$ uniform random elements in $\{\mbf{w}_1,\ldots,\mbf{w}_\ell\}$.

For each $g \in G$, write
\[
\mbf{u}_g = (\sang{g\mbf{v},\mbf{w}_1},\ldots,\sang{g\mbf{v},\mbf{w}_\ell})
\]
and
\[
\mbf{v}'
= (\snorm{\proj_{V_1}\mbf{v}}_2, \ldots, \snorm{\proj_{V_\ell}\mbf{v}}_2).
\]
Given $S\subseteq[\ell]$, let $\on{proj}_S$ take the projection of an $\ell$-dimensional vector down to that subset of coordinates. We have
\begin{align*}
\int\sup_{g\in G}\snorm{\proj_W(g\mbf{v})}_2^2d\mu(W) &= \frac{1}{\binom{\ell}{k}}\sum_{S\in\binom{[\ell]}{k}}\sup_{g\in G}\snorm{\on{proj}_S(\mbf{u}_g)}_2^2\\
&\le\frac{1}{\binom{\ell}{k}}\sum_{S\in\binom{[\ell]}{k}}\sum_{j\in S}(v_j')^2\\
&=\frac{k}{\ell}\sum_{j=1}^\ell (v_j')^2\le\frac{k}{\ell}\snorm{\mbf{v}}_2^2.
\end{align*}
The first equality follows by the definition of $\mu$, the subsequent inequality follows by $|\sang{g\mbf{v},\mbf{w}_j}|\le v_j'$, and the last line is by direct computation and orthogonality of the $V_j$.
\end{proof}

We next reduce the irreducible case to the primitive case. We first collect a few facts proved in \cite{Gre20} regarding systems of imprimitivity.

\begin{lemma}[{\cite[Section~2]{Gre20}}]\label{lem:system-of-imprimitivity}
Let $G\leqslant\msf{U}(\mb{C}^d)$ be irreducible but imprimitive. Consider a system imprimitivity
\[\mb{C}^d = \bigoplus_{j=1}^{d_1}V_j\]
with $d_1$ maximal over all such systems of primitivity. Let $H = \{g\in G: gV_1 = V_1\}$ and choose $\gamma_1,\ldots,\gamma_{d_1}$ such that $\gamma_jV_1 = V_j$. Then the following hold:
\begin{enumerate}[\qquad 1.]
    \item The $V_j$ are orthogonal and have the same dimension, and $G$ acts transitively on them.
    \item $H$ has primitive action on $V_1$ (i.e. the representation of $H$ on $V_1$ is primitive).
    \item $\gamma_1,\ldots,\gamma_{d_1}$ form a complete set of left coset representatives for $H$ in $G$.
    \item For each $g\in G$ there is $\sigma_g\in\mf{S}_{d_1}$ so that $\gamma_{\sigma_g(j)}^{-1}g\gamma_j\in H$ for all $j \in [d_1]$ (i.e., $\sigma_g$ records how $g$ permutes $\{V_1, \ldots, V_{d_1}\}$).
\end{enumerate}
\end{lemma}

Now we are ready to prove \cref{prop:reduction-prim}, which recall says that for all $k\le d/2$,
\[
f^{\on{irred}}_k(d)\le\max_{d_1d_2=d}\bigl(\min\bigl\{f^{\on{prim}}_{\lceil k/d_1\rceil}(d_2),f^{\on{alt}}_k(d_1)+\mbm{1}_{k\ge d_1}\bigr\}\bigr).
\]

\begin{proof}[Proof of \cref{prop:reduction-prim}]
Let $G\leqslant\msf{U}(\mb{C}^d)$ be irreducible but imprimitive. Consider a system of imprimitivity
\[\mb{C}^d = \bigoplus_{j=1}^{d_1}V_j\]
with $d_1$ maximal among all systems of imprimitivity. By \cref{lem:system-of-imprimitivity}, the spaces $V_j$ are orthogonal and all the $\dim V_j$ are equal. Let $d_2 = \dim V_1$, so that $d_1d_2 = d$. Furthermore, $H = \{g\in G: gV_1 = V_1\}$ acts primitively on $V_1$, that $G$ acts transitively on the $V_j$, and that there are $\gamma_1,\ldots,\gamma_{d_1}$ so that $\gamma_jV_1 = V_j$ which form a complete set of left coset representatives for $H$ in $G$. For each $g\in G$ we have some $\sigma_g\in\mf{S}_{d_1}$ so that $\gamma_{\sigma_g(j)}^{-1}g\gamma_j\in H$ for all $j\in[d_1]$. Define $h(g,j) = \gamma_{\sigma_g(j)}^{-1}g\gamma_j$.

Let $\mbf{v}\in\mb{C}^d$. There is a unique orthogonal decomposition
\[\mbf{v} = \sum_{j=1}^{d_1}\gamma_j\mbf{v}_j\]
where $\mbf{v}_j\in V_1$ for all $j\in[d_1]$. We have
\[g\mbf{v} = \sum_{j=1}^{d_1}g\gamma_j\mbf{v}_j = \sum_{j=1}^{d_1}\gamma_jh(g,\sigma_g^{-1}(j))\mbf{v}_{\sigma_g^{-1}(j)}.\]
Finally, if
\[\mbf{w} = \sum_{j=1}^{d_1}\lambda_j\gamma_j\mbf{x}\]
for some $\bm{\lambda} = (\lambda_1, \ldots, \lambda_{d_1}) \in\mb{C}^{d_1}$ and $\mbf{x}\in V_1$ then we see from the above and orthogonality that
\[\sang{g\mbf{v},\mbf{w}} = \sum_{j=1}^{d_1}\lambda_j\sang{h(g,\sigma_g^{-1}(j))\mbf{v}_{\sigma_g^{-1}(j)},\mbf{x}}.\]

Now we return to the situation at hand: we need to choose a $k$-dimensional space with a good projection for our transitive set. Consider the map $\psi: V_1\times\mb{C}^{d_1}\to\mb{C}^d$ given by
\[\psi(\mbf{x},\bm{\lambda}) = \sum_{j=1}^{d_1}\lambda_j\gamma_j\mbf{x}.\]
It clearly maps the pair of unit spheres into the unit sphere. Given probability measures $\mu_1$ on $\on{Gr}_{\mb{C}}(k_1,V_1)$ and $\mu_2$ on $\on{Gr}_{\mb{C}}(k_2,\mb{C}^{d_1})$, we define the pushforward measure $\mu$ on $\on{Gr}_{\mb{C}}(k_1k_2,d)$ by taking the image of these two subspaces under $\psi$.
Equivalently, suppose $\mu_1^\ast$ samples a unitary basis $\mbf{x}_1,\ldots,\mbf{x}_{k_1}$ of a subspace of $V_1$ and $\mu_2^\ast$ samples a unitary basis $\bm{\lambda}_1,\ldots,\bm{\lambda}_{k_2}$ of a subspace of $\mb{C}^{d_1}$, then $\mu$ samples the subspace of $\mbf{C}^d$ with basis $\{\psi(\mbf{x}_i,\bm{\lambda}_j): i\in[k_1], j\in[k_2]\}$. It is easy to check this basis is in fact unitary.

Next, we choose $\mu_1$ and $\mu_2$ based on the sizes of $d_1$ and $d_2$.

First let $k_1 = \lceil k/d_1\rceil\le d_2$ (as $k\le d/2$) and $k_2 = d_1$. We let $\mu_1$ be the measure guaranteed by \cref{def:k-dimensional-f} so that
\[\int\sup_{h\in H}\snorm{\proj_W(h\mbf{u})}_2^2d\mu_1(W)\le f^{\on{prim}}_{k_1}(d_2)^2\snorm{\mbf{u}}_2^2\]
for all $\mbf{u}\in V_1$ and let $\mu_2$ be the atom on the space $\mb{C}^{d_1}$ in $\on{Gr}_{\mb{C}}(d_1,d_1)$. Let $\mu$ be the $\psi$-pushforward of $(\mu_1, \mu_2)$ as described earlier. We find
\begin{align*}
\int\sup_{g\in G}\snorm{\proj_W(g\mbf{v})}_2^2d\mu(W) &= \int\sup_{g\in G}\sum_{\ell=1}^{k_1}\sum_{j=1}^{d_1}|\sang{g\mbf{v},\psi(\mbf{x}_\ell,\mbf{e}_j)}|^2d\mu_1^\ast(\mbf{x}_1,\ldots,\mbf{x}_{k_1})\\
&= \int\sup_{g\in G}\sum_{\ell=1}^{k_1}\sum_{j=1}^{d_1}|\sang{h(g,\sigma_g^{-1}(j))\mbf{v}_{\sigma_g^{-1}(j)},\mbf{x}_\ell}|^2d\mu_1^\ast(\mbf{x}_1,\ldots,\mbf{x}_{k_1})\\
&= \int\sup_{g\in G}\sum_{\ell=1}^{k_1}\sum_{j=1}^{d_1}|\sang{h(g,j)\mbf{v}_j,\mbf{x}_\ell}|^2d\mu_1^\ast(\mbf{x}_1,\ldots,\mbf{x}_{k_1})\\
&\le\sum_{j=1}^{d_1}\int\sup_{g\in G}\sum_{\ell=1}^{k_1}|\sang{h(g,j)\mbf{v}_j,\mbf{x}_\ell}|^2d\mu_1^\ast(\mbf{x}_1,\ldots,\mbf{x}_{k_1})\\
&\le\sum_{j=1}^{d_1}\int\sup_{h\in H}\sum_{\ell=1}^{k_1}|\sang{h\mbf{v}_j,\mbf{x}_\ell}|^2d\mu_1^\ast(\mbf{x}_1,\ldots,\mbf{x}_{k_1})\\
&\le\sum_{j=1}^{d_1}f^{\on{prim}}_{k_1}(d_2)^2\snorm{\mbf{v}_j}_2^2 = f^{\on{prim}}_{k_1}(d_2)^2\snorm{\mbf{v}}_2^2.
\end{align*}
The last equality is by orthogonality of $V_1,\ldots,V_{d_1}$ and unitarity of $\gamma_j$ for $j\in[d_1]$.

Now suppose that $k < d_1$. Let $k_1 = 1$ and $k_2 = k$. Choose an arbitrary unit vector $\mbf{x} \in V_1$ and $\mu_1$ be an atom on $\on{Gr}_{\mb{C}}(1,V_1)$ supported on the line $\mb{C}\mbf{x}$. Let $\mu_2$ be guaranteed by \cref{def:alternative-f} so that
\[\int\sup_{\mbf{u}\in\on{Dom}(\mbf{w})}\sum_{\ell=1}^k|\sang{\mbf{u},\bm{\lambda}_\ell}|^2d\mu_2^*(\bm{\lambda}_1,\ldots,\bm{\lambda}_k)\le f_k^{\on{alt}}(d_1)^2\snorm{\mbf{w}}_2^2\]
for all $\mbf{w}\in V_1$. 
Let $\mu$ be the $\psi$-pushforward of $(\mu_1, \mu_2)$ as described earlier. 
We find
\begin{align*}
\int\sup_{g\in G}\snorm{\proj_W(g\mbf{v})}_2^2d\mu(W) &= \int\sup_{g\in G}\sum_{\ell=1}^k|\sang{g\mbf{v},\mbf{w}_\ell}|^2d\mu^\ast(\mbf{w}_1,\ldots,\mbf{w}_k)\\
&= \int\sup_{g\in G}\sum_{\ell=1}^k|\sang{g\mbf{v},\psi(\mbf{x},\bm{\lambda}_\ell)}|^2d\mu_2^\ast(\bm{\lambda}_1,\ldots,\bm{\lambda}_k)\\
&=\int\sup_{g\in G}\sum_{\ell=1}^k\bigg|\sum_{j=1}^{d_1}\lambda_{\ell,j}\sang{h(g,\sigma_g^{-1}(j))\mbf{v}_{\sigma_g^{-1}(j)},\mbf{x}}\bigg|^2d\mu_2^\ast(\bm{\lambda}_1,\ldots,\bm{\lambda}_k)\\
&\le\int\sup_{\mbf{u}\in\on{Dom}(\mbf{y})}\sum_{\ell=1}^k|\sang{\mbf{u},\bm{\lambda}_\ell}|^2d\mu_2^\ast(\bm{\lambda}_1,\ldots,\bm{\lambda}_k),
\end{align*}
where $\mbf{y}$ has coordinates $y_j = \sup_{h\in H}|\sang{h\mbf{v}_j,\mbf{x}}|$ for $j\in[d_1]$. We immediately deduce
\begin{align*}
\int\sup_{g\in G}\snorm{\proj_W(g\mbf{v})}_2^2d\mu(W)&\le\int\sup_{\mbf{u}\in\on{Dom}(\mbf{y})}\sum_{\ell=1}^k|\sang{\mbf{u},\bm{\lambda}_\ell}|^2d\mu_2^\ast(\bm{\lambda}_1,\ldots,\bm{\lambda}_k)\\
&\le f_k^{\on{alt}}(d_1)^2\snorm{\mbf{y}}_2^2\le f_k^{\on{alt}}(d_1)^2\sum_{j=1}^{d_1}\snorm{\mbf{v}_j}_2^2 = f_k^{\on{alt}}(d_1)^2\snorm{\mbf{v}}_2^2.
\end{align*}
Note that the above constructed measures in both cases are independent of $\mbf{v}$. The second construction is only valid when $k < d_1$. Therefore since the $f$ values are clearly bounded by $1$, we have an upper bound of
\[f^{\on{prim}}_k(d)\le\max_{d_1d_2=d}(\min(f^{\on{prim}}_{\lceil k/d_1\rceil}(d_2),f^{\on{alt}}_k(d_1)+\mbm{1}_{k\ge d_1})),\]
as claimed.
\end{proof}

\section{Permutation groups}\label{sec:permutation}
In this section, we establish upper bounds for $f^{\on{sym}}_k(d)$ and $f^{\on{alt}}_k(d)$, extending the previous construction \cite[Section~3]{Gre20} for $k=1$.

A useful high dimensional intuition is that, for small $k$, a random $k$-dimensional subspace of $\mb{R}^d$ has the property that \emph{all} its unit vectors have distribution of coordinate magnitudes similar to that of a random Gaussian vector.

We first need the existence of a large dimension subspace of $\mb{R}^d$ with certain delocalization properties. We encode this through the following norm.
\begin{definition}\label{def:strange-norm}
Given $\mbf{v}\in\mb{R}^d$, let
\[\snorm{\mbf{v}}_T^2 = \sup_{\emptyset\subsetneq S\subseteq[d]}\log^4(2d/|S|)\sum_{j\in S}v_j^2\]
and let
\[T^\ast = \{\mbf{t}\in\mb{R}^d: |\sang{\mbf{t},\mbf{w}}|\le 1 \text{ whenever } \snorm{\mbf{w}}_T\le 1\}.\]

\end{definition}
\begin{remark}
Note that $\snorm{\cdot}_T$ is a norm as it can be represented as a supremum of seminorms. Hence
\[\snorm{\mbf{w}}_T = \sup_{t\in T^\ast}|\sang{\mbf{t},\mbf{w}}|.\]
\end{remark}

We next recall a classical lemma regarding the concentration of norms on Gaussian space (see e.g. \cite{LT91}); we provide a short proof for convenience.
\begin{lemma}\label{lem:concentration}
There is an absolute constant $C > 0$ so that for all $p\ge 1$, a Gaussian random vector $\mbf{w}\sim\mc{N}(0,I_d)$ satisfies
\[(\mb{E}_{w_1+\cdots+w_d=0}\snorm{\mbf{w}}_T^p)^{1/p}\le (\mb{E}\snorm{\mbf{w}}_T^p)^{1/p}\le\mb{E}\snorm{\mbf{w}}_T + C\sqrt{p}\sup_{\mbf{t}\in T^\ast}\snorm{\mbf{t}}_2.\]
\end{lemma}
\begin{proof}
For first inequality note that $\mbf{w}\sim \mc{N}(0,I_d)$ can be written as $\mbf{w'} + G\mbf{1}$ where $\mbf{w'}$ is drawn from $\mc{N}(0,I_d)$ conditioned on having coordinate sum zero and $G \in \mc{N}(0,1)$ is independent of $\mbf{w'}$. Then by convexity note that 
\[(\mb{E}\snorm{\mbf{w}}_T^p)^{1/p} = (\mb{E}\snorm{\mbf{w'} + G\mbf{1}}_T^p)^{1/p}\ge (\mb{E}_{\mbf{w'}}\snorm{\mb{E}[\mbf{w'} + \mbf{v'}|\mbf{w'}]}_T^p)^{1/p} = (\mb{E}_{w_1+\cdots+w_d=0}\snorm{\mbf{w}}_T^p)^{1/p}.\]
To prove the second inequality first note that 
\[\snorm{\mbf{w}}_T-\snorm{\mbf{v}}_T\le\snorm{\mbf{w}-\mbf{v}}_T = \sup_{\mbf{t}\in T*}|\sang{\mbf{t},\mbf{w}-\mbf{v}}|\le \snorm{\mbf{w}-\mbf{v}}_2\sup_{t\in T*}\snorm{\mbf{t}}_2.\]
Therefore if $L = \sup_{\mbf{t}\in T^\ast}\snorm{\mbf{t}}_2$ then $\mbf{w}\mapsto\snorm{\mbf{w}}_T$ is an $L$-Lipschitz function with respect to Euclidean distance. Therefore by Gaussian concentration for Lipschitz functions (see e.g. \cite[p.~125]{BLM13}) we have that 
\[\mb{P}[|\snorm{\mbf{w}}_T-\mb{E}[\snorm{\mbf{w}}_T]|\ge t]\le 2\exp(-ct^2/L^2)\]
where $c$ is an absolute constant. Using standard moment bounds for sub-Gaussian random variables (see e.g.~\cite[Proposition~2.5.2]{Ver18}), we find that 
\[(\mb{E}|\snorm{\mbf{w}}_T-\mb{E}\snorm{\mbf{w}}_T|^p)^{1/p}\le C\sqrt{p}\sup_{\mbf{t}\in T^\ast}\snorm{\mbf{t}}_2\]
for an absolute constant $C > 0$. Finally, Minkowski's inequality implies that 
\[(\mb{E}\snorm{\mbf{w}}_T^p)^{1/p}\le \mb{E}\snorm{\mbf{w}}_T+ (\mb{E}|\snorm{\mbf{w}}_T-\mb{E}\snorm{\mbf{w}}_T|^p)^{1/p}\]
and therefore the result follows.
\end{proof}

We now prove an upper bound for $\mb{E}[\snorm{\mbf{w}}_T]$.
\begin{lemma}\label{lem:expectation-strange-norm}
A Gaussian random vector $\mbf{w} \sim \mc{N}(0,I_d)$ satisfies $\mb{E}\snorm{\mbf{w}}_T\lesssim\sqrt{d}$.
\end{lemma}
\begin{proof}
Recall $\mbf{w}_i^\succ$ from \cref{def:decreasing-form}.
We have
\begin{align*}
\mb{E}(w_i^\succ)^2 &= \int_0^\infty\mb{P}[w_i^\succ\ge\sqrt{t}]dt\le\int_0^\infty\min\bigg(1,\binom{d}{i}(2e^{-t/2})^i\bigg)dt\\
&\le\int_0^\infty\min(1,(2de^{1-t/2}/i)^i)dt\lesssim\log(2d/i).
\end{align*}
Therefore
\begin{align*}
(\mb{E}\snorm{\mbf{w}}_T)^2&\le\mb{E}\snorm{\mbf{w}}_T^2\le\sum_{i=1}^d\log^4(2d/i)(w_i^\succ)^2\lesssim\sum_{i=1}^d\log^5(2d/i)\\
&\le d\int_0^1\log(2/x)^5~dx = d\int_0^\infty(y+\log 2)^5e^{-y}~dy\lesssim d.\qedhere
\end{align*}
\end{proof}
We are in position to derive a high-probability version.
\begin{lemma}\label{lem:gaussian-strange-norm}
With probability at least $1 - \exp(-2d/(\log d)^4)$, a standard Gaussian vector $\mbf{w} \sim \mc{N}(0,I_d)$ satisfies $\snorm{\mbf{w}}_T\lesssim\sqrt{d}$. In fact, the same is true after conditioning $\mbf{w}$ to have coordinate sum $0$.
\end{lemma}
\begin{proof}
Note that if $\mbf{t}\in T^\ast$, then
\[\snorm{\mbf{t}}_2^2 = \snorm{\mbf{t}}_T\bigg|\bigg\langle\mbf{t},\frac{\mbf{t}}{\snorm{\mbf{t}}_T}\bigg\rangle\bigg|\le\snorm{\mbf{t}}_T\le\log^2(2d)\snorm{\mbf{t}}_2.\]
Hence
\[\sup_{\mbf{t}\in T^\ast}\snorm{\mbf{t}}_2\le\log^2(2d).\]
To deduce the claimed bound, note that
\begin{align*}
\mb{P}[\snorm{\mbf{w}}_T\ge K\sqrt{d}]&\le (K\sqrt{d})^{-p}\mb{E}[\snorm{\mbf{w}}_T^p]
\\
&\le(K\sqrt{d})^{-p}(\mb{E}\snorm{\mbf{w}}_T+C\sqrt{p}\sup_{\mbf{t}\in T^\ast}\snorm{\mbf{t}}_2)^p\\
&\le(K\sqrt{d})^{-p}(C'\sqrt{d}+C'\sqrt{p}\log^2(2d))^p
\end{align*}
for appropriate absolute constants $C,C' > 0$, using \cref{lem:concentration,lem:expectation-strange-norm} and the above inequality. Letting $p = d/(\log d)^4$ and $K > 0$ be a sufficiently large absolute constant yields
\[\mb{P}[\snorm{\mbf{w}}_T\ge K\sqrt{d}]\le\exp(-2p),\]
as desired. The same holds is we condition on sum $0$, using the moment bound for the conditional variable derived in \cref{lem:concentration} instead.
\end{proof}

\begin{lemma}\label{lem:subspace-strange-norm}
There is a $\lceil d/(\log d)^4\rceil$-dimensional subspace of the hyperplane $\mbf{1}^\perp$ in $\mb{R}^d$ such that each of its unit vectors $\mbf{v}$ satisfies
\[\snorm{\mbf{v}}_T\lesssim 1.\]
\end{lemma}
\begin{proof}
We can assume $d$ is sufficiently large. Let $k = \lceil d/(\log d)^4\rceil$,
and consider a uniformly random $k$-dimensional subspace $W$ of $\mbf{1}^\perp$. Let $U$ be a $d \times k$ matrix whose columns form an orthonormal basis of $W$, chosen uniformly at random.

By a standard volume packing argument (e.g., see \cite[Lemma~4.3]{Rud14}), there exists $\mc{N} \subset \mb{S}(\mb{R}^k)$ with $\sabs{\mc{N}} \le 6^k$ such that for every $\mbf{v}\in\mb{S}(\mb{R}^k)$ there is $\mbf{v}'\in\mc{N}$ so that $\snorm{\mbf{v}-\mbf{v}'}_2\le 1/2$. Thus if $\mbf{u}$ is a unit vector in the direction of $\mbf{v}-\mbf{v}'$, we have
\[\snorm{U\mbf{v}}_T\le\snorm{U\mbf{v}'}_T + \snorm{U(\mbf{v}-\mbf{v}')}_T\le\snorm{U\mbf{v}'}_T + \frac{1}{2}\snorm{U\mbf{u}}_T.\]
We deduce
\[\sup_{\mbf{v}\in\mb{S}(\mb{R}^k)}\snorm{U\mbf{v}}_T\le\sup_{\mbf{v}'\in\mc{N}}\snorm{U\mbf{v}'}_T + \frac{1}{2}\sup_{\mbf{u}\in\mb{S}(\mb{R}^k)}\snorm{U\mbf{u}}_T\]
and thus
\[\sup_{\mbf{v}\in\mb{S}(\mb{R}^k)}\snorm{U\mbf{v}}_T\le 2\sup_{\mbf{v}'\in\mc{N}}\snorm{U\mbf{v}'}_T.\]

Now fix some $\mbf{v}\in\mc{N}$. Note the distribution of $U\mbf{v}$ is uniform among unit vectors in $\mbf{1}^\perp$ since $W$ was chosen uniformly. Now note that for any constant $C$ we have that 
\[\mb{P}[\snorm{U\mbf{v}}_T\ge C] = \mb{P}[\snorm{\mbf{G}/\snorm{\mbf{G}}_2}_T\ge C]\]
where $\mbf{G}\sim N(0,I_d-(\mbf{1}^T\mbf{1})/d)$. Now since $\mbf{G}/\snorm{\mbf{G}}_2$ and $\snorm{\mbf{G}}_2$ are independent we have that
\begin{align*}
\mb{P}[\snorm{\mbf{G}/\snorm{\mbf{G}}_2}_T\ge C] &= \mb{P}[\snorm{\mbf{G}}_2\le 2\sqrt{d}]^{-1}\mb{P}[\snorm{\mbf{G}/\snorm{\mbf{G}}_2}_T\ge C \text{ and } \snorm{\mbf{G}}_2\le 2\sqrt{d}]\\
&\le 2\mb{P}[\snorm{\mbf{G}/\snorm{\mbf{G}}_2}_T\ge C \text{ and } \snorm{\mbf{G}}_2\le 2\sqrt{d}]\\
&\le 2\mb{P}[\snorm{\mbf{G}/\snorm{\mbf{G}}_2}_T\ge 2C\sqrt{d}].
\end{align*}
By \cref{lem:gaussian-strange-norm}, the last expression is at most $2\exp(-2d/(\log d)^4)$. The result follows upon taking the union bound over at most $6^k$ vectors in $\mc{N}$, since $6 < e^2$.
\end{proof}

Finally, we will need a form of Selberg's inequality (see \cite[Chapter~27,~Theorem~1]{Dav00}).
\begin{lemma}\label{lem:selberg}
For $\mbf{v}_1,\ldots,\mbf{v}_m\in \mb{C}^d$ we have that
\[\sup_{\mbf{w}\in\mb{S}(\mb{C}^d)}\sum_{i=1}^m|\sang{\mbf{w},\mbf{v}_i}|^2\le \sup_{i\in[m]}\sum_{j=1}^m|\sang{\mbf{v}_i,\mbf{v}_j}|.\]
\end{lemma}

Now we prove \cref{prop:permutation}, which recall says that for $k\le d/(\log d)^5$, one has
\[
f_k^{\on{sym}}(d)\le f_k^{\on{alt}}(d)\lesssim 1/\sqrt{\log(d/k)}.
\]
The first inequality is immediate as the set of allowable $\mu$'s in the definition of $f_k^{\on{alt}}$ is a subset of those of $f_k^{\on{sym}}$. So we just need to prove the second inequality.

\begin{proof}[Proof of \cref{prop:permutation}]
Let $\mbf{e}_i$ be the $i$-th coordinate vector. For each $j$ with $k\le 2^j/(\log 2^j)^4\le d$, we apply \cref{lem:subspace-strange-norm} to the space $V_j = \on{span}_{\mb{R}}\{\mbf{e}_1,\ldots,\mbf{e}_{2^j}\}$. Here the $T$-norm is defined with respect to this $2^j$-dimensional space. In particular, there exists a $k$-dimensional (real) subspace of the orthogonal complement of $\mbf{e}_1+\cdots+\mbf{e}_{2^j}$ within $V_j$, call it $W_j$, so that every unit vector $\mbf{u}\in W_j$ satisfies
\[\sum_{i\in S}u_i^2\lesssim\frac{1}{\log^4(2^{j+1}/|S|)}\]
for every nonempty $S\subseteq[2^j]$. Let $V_j' = \on{span}_{\mb{C}} V_j$ and $W_j' = \on{span}_{\mb{C}} W_j$. We immediately deduce that every unit vector $\mbf{u}\in W_j'$ satisfies
\begin{equation}\label{eq:delocalized}
\sum_{i\in S}|u_i|^2\lesssim\frac{1}{\log^4(2^{j+1}/|S|)}
\end{equation}
since we can write it as $\mbf{u} = \alpha\mbf{u}_r + \beta\sqrt{-1}\mbf{u}_c$ where $\mbf{u}_r,\mbf{u}_c\in W_j$ are real unit vectors and $\alpha,\beta\in\mb{R}$ satisfy $\alpha^2+\beta^2 = 1$.

Now we construct our random subspace as follows: let $W = W_j$ where $j$ is a random integer uniformly chosen from
\[
J = \{\lceil\log_2(2k\log^4 d)\rceil,\ldots, \lfloor \log_2 d \rfloor\}.
\]
Let $\mu$ be the probability measure on $\Gr_\mb{C}(k, n)$ that gives $W$.

For every $\mbf{v}\in\mb{S}(\mb{C}^d)$, we have
\[\sup_{\gamma\in\Gamma_d}\snorm{\proj_W(\gamma\mbf{v})}_2 = \sup_{\gamma\in\Gamma_d}\sup_{\mbf{w}\in\mb{S}(W)}|\sang{\gamma\mbf{v},\mbf{w}}|= \sup_{\mbf{w}\in\mb{S}(W)}\sang{\mbf{v}^\succ,\mbf{w}^\succ}.\]
Therefore
\[\int\sup_{\gamma\in\Gamma_d}\snorm{\proj_W(\gamma\mbf{v})}_2^2d\mu(W) = \frac{1}{|J|}\sum_{j\in J}\sup_{\gamma\in\Gamma_d}\snorm{\proj_{W_j}(\gamma\mbf{v})}_2^2 = \frac{1}{|J|}\sum_{j\in J}\sup_{\mbf{w}\in\mb{S}(W_j)}\sang{\mbf{v}^\succ,\mbf{w}^\succ}^2.\]
Let $\mbf{w}_j'\in\mb{S}(W_j)$ be such that
\[\sup_{\mbf{w}\in\mb{S}(W_j)}\sang{\mbf{v}^\succ,\mbf{w}^\succ}^2 = \sang{\mbf{v}^\succ,(\mbf{w}_j')^\succ}^2,\]
which exists by compactness. For $i,j\in J$ with $i\ge j$, we have
\[|\sang{(\mbf{w}_i')^\succ,(\mbf{w}_j')^\succ}|\le \snorm{\proj_{V_i}((\mbf{w}_j')^\succ)}_2 \lesssim \frac{1}{\log^2(2^{i+1}/2^j)}.\]
The first inequality follows from $\mbf{w}_j'\in V_j$, which implies $(\mbf{w}_j')^\succ\in V_j$. The second follows from \cref{eq:delocalized} applied to $\mbf{w}_j'$ and $S$ a subset of $[2^i]$ composed of the $2^j$ largest magnitude coordinates of $\mbf{w}_j'$.

Applying \cref{lem:selberg}, we deduce
\begin{align*}
\int\sup_{\gamma\in\Gamma_d}\snorm{\proj_W(\gamma\mbf{v})}_2^2d\mu(W)&= \frac{1}{|J|}\sum_{j\in J}\sang{\mbf{v}^\succ,(\mbf{w}_j')^\succ}^2\le\sup_{i\in J}\frac{1}{|J|}\sum_{j\in J}|\sang{(\mbf{w}_i')^\succ,(\mbf{w}_j')^\succ}|\\
&\lesssim\frac{1}{|J|}\bigg(\sum_{j\in J, j\ge i}\frac{1}{\log^2(2^{j+1}/2^i)} + \sum_{j\in J, j < i}\frac{1}{\log^2(2^{i+1}/2^j)}\bigg)\lesssim\frac{1}{|J|}.
\end{align*}
This $\mu$ thus shows that $f_k^{\on{alt}}(d) \lesssim 1/\sqrt{\log(d/k)}$.
\end{proof}

\section{Primitive representations}\label{sec:primitive}
We now turn to the case of bounding $f_k^{\on{prim}}(d)$. 
First, we show that if the group $G \leqslant \msf{U}(\mb{R}^d)$ is sufficiently small, then a random basis achieves the necessary bound for $f_k^{\on{prim}}(d)$. This is a minor modification of \cite[Proposition~4.1]{Gre20}.

\begin{proposition}\label{thm:few-points}
Let $G\leqslant\msf{U}(\mb{C}^d)$. Suppose that $[G:Z_d\cap G]\le e^{d/\log d}$, where $Z_d :=\{\lambda I_d: |\lambda|=1\}$. Then for $k\in[d]$ there exists a probability measure $\mu$ on $\on{Gr}_\mb{C}(k,d)$ such that
\[\int\sup_{g\in G}\snorm{\proj_W(g\mbf{v})}_2^2d\mu(W)\lesssim  \frac{1}{\log(2d/k)}\snorm{\mbf{v}}_2^2\]
for all $\mbf{v}\in\mb{C}^d$.
\end{proposition}

\begin{proof}
We let $\mu$ be the uniform measure on $\Gr_\mb{C}(k,d)$.
By scaling, we may assume that $\mbf{v}$ is a unit vector. Furthermore let $W'$ be the subspace generated by the first $k$ coordinate vectors $\mbf{e}_1,\ldots,\mbf{e}_k$. 
Note that
\begin{align*}
\mb{P}_W\bigg[\sup_{g\in G}\snorm{\proj_W(g\mbf{v})}_2\ge t\bigg]&\le e^{d/\log d}\mb{P}_W[\snorm{\proj_W(\mbf{v})}_2\ge t]\\
&\le e^{d/\log d}\mb{P}_{\mbf{v}'\in\mb{S}(\mb{C}^d)}[\snorm{\proj_{W'}(\mbf{v}')}_2\ge t]
\end{align*}
using a union bound and then orthogonal invariance. Now note that 
\[\mb{E}[\snorm{\proj_{W'}(\mbf{v}')}_2]^2\le\mb{E}[\snorm{\proj_{W'}(\mbf{v}')}_2^2] = k/d\]
and that $\snorm{\proj_{W'}(\mbf{v'})}$ is a $1$-Lipschitz function of $\mbf{v}'$. Therefore by L\'evy concentration on the sphere we have that 
\[\mb{P}_{\mbf{v}'\in\mb{S}(\mb{C}^d)}[\snorm{\proj_{W'}(\mbf{v}')}_2\ge\sqrt{k/d}+C/\sqrt{\log d}]\le e^{-2d/\log d}\]
for a suitably large absolute constant $C$. Finally, using $\sqrt{k/d}\lesssim 1/\sqrt{\log(2d/k)}$ and using the bound $\snorm{\proj_{W'}(\mbf{v'})}_2\le 1$, the desired result follows immediately.
\end{proof}

We need the following key group theoretic result from Green \cite{Gre20}, which in turn builds on ideas from Collins' work on optimal bounds for Jordan's theorem~\cite{Col07}.
Roughly, it says that if $[G:Z_d\cap G]$ is large then $G$ has a large normal alternating subgroup. The first part of the following theorem is \cite[Proposition~4.2]{Gre20}, while the rest is implicit in the proof of \cite[Proposition~1.11]{Gre20}.

\begin{theorem}[{\cite[Section~4]{Gre20}}]\label{thm:group-theory}
Let $G\leqslant\msf{U}(\mb{C}^d)$ be primitive and suppose that $[G:Z_d\cap G]\ge e^{d/\log d}$. If $d$ is sufficiently large then all of the following hold.
\begin{enumerate}
    \item $G$ has a normal subgroup isomorphic to the alternating group $A_n$ for some $n\gtrsim d/(\log d)^4$.
    \item $G$ has a subgroup of index at most $2$ of the form $A_n\times H$, with the same $n$.
    \item The resulting representation $\rho: A_n\times H\hookrightarrow G\hookrightarrow\msf{U}(\mb{C}^d)$ decomposes into irreducible representations, at least one of which (call it $\rho_1$) is of the form $\rho_1\simeq\psi\otimes\psi'$, where $\psi'$ is an irreducible representation of $H$ and $\psi$ is the representation of $A_n$ acting via permutation of coordinates on $\{\mbf{z}\in\mb{C}^n: z_1+\cdots+z_n=0\}$.
\end{enumerate}
\end{theorem}

We are now in position to prove \cref{prop:f-primitive}, which recall says that there is an absolute constant $c > 0$ such that for every $k\le cd/(\log d)^4$ we have 
\[
f_k^{\on{prim}}(d)\lesssim\sup_{d'\ge cd/(\log d)^4}f^{\on{alt}}_k(d').
\]
The proof mirrors that of \cite[Proposition~1.11]{Gre20}, but we correct an error of Green (\cite[p.~20]{Gre20}) involving an incorrect orthogonality identity. This erroneous deduction is replaced by an argument which still allows one to reduce the primitive case to the alternating case.

\begin{proof}[Proof of \cref{prop:f-primitive}]
We may assume $d$ is sufficiently large.
If $[G:Z_d\cap G]\le e^{d/\log d}$, then the result follows by \cref{thm:few-points}.
So we can assume $[G:Z_d\cap G]\ge e^{d/\log d}$, and thus by \cref{thm:group-theory}, $G$ has a normal subgroup isomorphic to $A_n$ for some $n\gtrsim d/(\log d)^4$ and that $G$ has a subgroup of index at most $2$ which is of the form $A_n\times H$. If the index is $2$, let $\tau$ be the nontrivial right coset representative of $A_n\times H$ in $G$ (otherwise just let $\tau$ be the identity). Note that
\begin{align*}
\sup_{g\in G}\snorm{\proj_W(g\mbf{v})}_2^2&\le\sup_{g\in A_n\times H}\snorm{\proj_W(g\mbf{v})}_2^2 + \sup_{g\in A_n\times H}\snorm{\proj_W(g\tau\mbf{v})}_2^2,
\end{align*}
so it is easy to see that, up to losing a constant factor, we may reduce to studying groups of the form $G = A_n\times H$ where $n\gtrsim d/(\log d)^4$ (but note that the representation may no longer be primitive, or even irreducible).

Now \cref{thm:group-theory} shows that the representation $\rho: A_n\times H\to\msf{U}(\mb{C}^d)$ coming from this setup has an irreducible component of the form $\rho_1\simeq\psi\otimes\psi'$, where $\psi'$ is an irreducible representation of $H$ and $\psi$ is the representation of $A_n$ acting via permutation of coordinates on $\{\mbf{z}\in\mb{C}^n: z_1+\cdots+z_n=0\}$.

Note that $\dim \rho_1 \ge \dim \psi = n-1\gtrsim d/(\log d)^4$, so $\dim\rho_1 \ge k$ provided that $c>0$ is sufficiently small. We will choose a $k$-dimensional subspace of the irreducible component $\rho_1$.

We explicitly present this situation as follows. Let $V'$ be the space acted on by $\psi'$ (unitarily). Consider $V = \mbf{1}^\perp\subseteq\mb{C}^n$, and consider the spaces $V\otimes V'\subseteq\mb{C}^n\otimes V'$, which has a natural unitary structure given by the tensor product. Note $\psi$ acts on $V$ by permutation of coordinates when represented in $\mb{C}^n$. Every vector in $V\otimes V'$ is spanned by pure tensors $\mbf{v}\otimes\mbf{v}'$ where $\mbf{v}$ has zero coordinate sum, and $\rho_1((a,h))$ acts by $\psi(a)\otimes\psi'(h)$ on pure tensors. In fact, we can extend this action to all of $\mb{C}^n\otimes V'$ in the natural way (and the resulting representation is isomorphic to a direct sum of $\rho_1$ and $\on{triv}_{A_n}\otimes\psi'$). At this point, the analysis will be similar to that in the proof of \cref{prop:reduction-prim}.

Let $\nu$ be the measure on $\on{Gr}_{\mb{C}}(k,n)$ which is guaranteed by \cref{def:alternative-f} (so is supported on subspaces of $V\subseteq\mb{C}^n$) and consider the measure which is supported on a single atom in $\on{Gr}_{\mb{C}}(1,V')$ in the direction of a fixed unit vector $\mbf{x}$. Let $\mu$ be the tensor of these two measures, i.e., if $\nu^\ast$ samples $k$ orthonormal (sum zero) vectors $\mbf{u}_1,\ldots,\mbf{u}_k$ then we choose the subspace with basis $\mbf{u}_1\otimes\mbf{x},\ldots,\mbf{u}_k\otimes\mbf{x}$.

Now consider some $\mbf{v}$ in the space $V\otimes V'\subseteq\mb{C}^n\otimes V'$, and write it as
\[\mbf{v} = \sum_{j=1}^n\mbf{e}_j\otimes\mbf{v}_j'\]
where the $\mbf{e}_j$ is the $j$-th coordinate vector of $\mb{C}^n$. In fact, the $\mbf{v}_j'$ must add up to $\mbf{0}\in V'$. We see that
\[\snorm{\mbf{v}}_2^2 = \sum_{j=1}^n\snorm{\mbf{v}_j'}_2^2.\]
We have
\begin{align*}
\int\sup_{g\in A_n\times H}&\snorm{\proj_W(g\mbf{v})}_2^2d\mu(W)\\
&= \int\sup_{a\in A_n,h\in H}\sum_{\ell=1}^k\bigg|\bigg\langle\sum_{j=1}^n\psi(a)\mbf{e}_j\otimes\psi'(h)\mbf{v}_j',\mbf{w}_\ell\bigg\rangle\bigg|^2 d\mu^\ast(\mbf{w}_1,\ldots,\mbf{w}_k)\\
&= \int\sup_{a\in A_n,h\in H}\sum_{\ell=1}^k\bigg|\sum_{j=1}^n\sang{\psi(a)\mbf{e}_j,\mbf{u}_\ell}\sang{\psi'(h)\mbf{v}_j',\mbf{x}}\bigg|^2d\nu^\ast(\mbf{u}_1,\ldots,\mbf{u}_k)\\
&\le\int\sup_{\mbf{w}\in\on{Dom}(\mbf{y})}\sum_{\ell=1}^k|\sang{\mbf{w},\mbf{u}_\ell}|^2d\nu^\ast(\mbf{u}_1,\ldots,\mbf{u}_k)\\
&\le f^{\on{alt}}_k(n)^2\snorm{\mbf{y}}_2^2,
\end{align*}
where $\mbf{y}\in\mb{C}^n$ satisfies $y_j = \sup_{h\in H}|\sang{\psi'(h)\mbf{v}_j',\mbf{x}}|$. The first inequality follows by noting that $\sang{\psi(a)\mbf{e}_j,\mbf{u}_\ell}$ as $j$ varies simply records the coordinates of $\mbf{u}_\ell$ in some permutation, and by considering $\mbf{w} = (w_1,\ldots,w_n)$ defined via $w_j = \sang{\psi'(h)\mbf{v}_j',\mbf{x}}$, which is clearly on $\on{Dom}(\mbf{y})$. Now  we see
\[\int\sup_{g\in A_n\times H}\snorm{\proj_W(g\mbf{v})}_2^2d\mu(W)\le f^{\on{alt}}_k(n)^2\snorm{\mbf{y}}_2^2\le f^{\on{alt}}_k(n)^2\sum_{j=1}^n\snorm{\mbf{v}_j'}_2^2 = f^{\on{alt}}_k(n)^2\snorm{\mbf{v}}_2^2.\qedhere\]
\end{proof}

This completes all the components of the proof of \cref{thm:complex}.

\section{Real subspaces}\label{sec:real}
We already proved \cref{thm:complex}, which finds a complex subspace. 
Now we use it to deduce \cref{thm:real}, which gives a real subspace.
We will apply the following version of the restricted invertibility theorem, which is a special case of \cite[Theorem~6]{NY17}. We write $s_1(M) \ge s_2(M) \ge \cdots$ for the singular values of a matrix $M$.

\begin{theorem}[{\cite[Theorem~6]{NY17}}]\label{thm:restricted-invert}
Let $M$ be a real $2k\times 4k$ matrix of rank $2k$. There exists $S\subseteq [4k]$ with $|S| = k$ such that $M_S$, the restriction of $M$ to the columns $S$, satisfies
\[s_k(M_S)\gtrsim \sqrt{\frac{\sum_{j=3k/2}^{4k} s_j(M)^2}{k}}.\]
\end{theorem}

\begin{proof}[Proof of \cref{thm:real}]
Let $2k\le d/(\log d)^C$, where $C$ is as in \cref{thm:complex}. By embedding $X$ in $\mb{S}(\mb{C}^d)$ and using \cref{thm:complex} we can find a $2k$-dimensional complex subspace $W$ of $\mb{C}^d$ such that 
\[\sup_{\mbf{x}\in X}\snorm{\proj_W\mbf{x}}_2\lesssim 1/\sqrt{\log (d/k)}.\]
Let $\mbf{v}_1,\ldots,\mbf{v}_{2k}$ be a unitary basis for the subspace $W$ and let the matrix with these columns be denoted by $B$. Now consider the matrix $M$ which has $4k$ columns which are $\on{Re}\mbf{v}_1,\ldots,\on{Re}\mbf{v}_{2k}$ and $\on{Im}\mbf{v}_1,\ldots,\on{Im}\mbf{v}_{2k}$. Note that $M$ has $s_{2k}(M)\ge 1/\sqrt{2}$ as any vectors in $\mb{C}^{4k}$ satisfying $iv_j = v_{j+2k}$ have 
$\snorm{M\mbf{v}} =  \snorm{\mbf{v}}/\sqrt{2}$. Therefore by \cref{thm:restricted-invert} one can select $k$ columns such that the matrix $N$ with those $k$ columns satisfies
\[s_k(N)\gtrsim 1.\]
Now consider any unit vector $\mbf{v}$ in the image of $N$. Such a vector can be represented as $\mbf{v} = N\mbf{w}$ where $\snorm{\mbf{w}}\lesssim 1$. It therefore suffices to prove that 
\[\sup_{\mbf{x}\in X,\mbf{w}\in \mb{S}(\mb{R}^k)}|\sang{N\mbf{w}, \mbf{x}}|\lesssim 1/\sqrt{\log (d/k)}.\]
To see this separate $N$ into $N_1$ and $N_2$ where $N_1$ corresponds to columns chosen from the real parts of vectors $\mbf{v}_i$ and the columns are chosen from the complex parts of $\mbf{v}_i$. Let these have $\ell$ and $k-\ell$ columns respectively. Then 
\begin{align*}
\sup_{\mbf{x}\in X,\mbf{w}\in \mb{S}(\mb{R}^k)}|\sang{N\mbf{w}, \mbf{x}}|&\le \sup_{\mbf{x}\in X,\mbf{w}\in \mb{S}(\mb{R}^\ell)}|\sang{N_1\mbf{w}, \mbf{x}}| + \sup_{\mbf{x}\in X,\mbf{w}\in \mb{S}(\mb{R}^{k-\ell})}|\sang{N_2\mbf{w}, \mbf{x}}|\\
&\le 2\sup_{\mbf{x}\in X,\mbf{w}\in \mb{S}(\mb{C}^k)}|\sang{B\mbf{w}, \mbf{x}}|\\
&\lesssim 1/\sqrt{\log (d/k)}.\qedhere
\end{align*}
\end{proof}

\section{Lower Bound}\label{sec:optimality}
Finally, we show a lower bound of $\Omega(1/\sqrt{\log(2d/k)})$, which demonstrates optimality of our results.

\begin{proof}[Proof of \cref{thm:lower-bound}]
We prove the real case; an analogous proof works over $\mb{C}$ by considering a suitably fine discretization of $\Gamma_d$, or we can repeat the proof in \cref{sec:real} to transfer a lower bound from real to complex. 

The claim for $k=1$ was already proved in \cite[\emph{Sharpness} after Theorem 1.3]{Gre20} (see the construction at the beginning of this article right after \cref{thm:green}).
The case $k=1$ implies the result also for $k \le d^{1-c}$ for any constant $c$, since we can project from $W$ onto a arbitrary 1-dimensional subspace of $W$.

So from now on assume $k\ge d^{1/2}$. Consider the action of $G = \mf{S}_d\ltimes(\mb{Z}/2\mb{Z})^d$ on $\mb{R}^d$ by permutation and signing. Let
\[\mbf{a} = \left( \frac{1}{\sqrt{\lfloor k/2 \rfloor+1}},\ldots, \frac{1}{\sqrt{d}},0,\ldots,0\right).\]
Let $X$ be the $G$-orbit of $\mbf{a}/\snorm{\mbf{a}}_2$.

Let $W$ be a $k$-dimensional subspace of $\mb{R}^d$. 
We wish to show $\sup_{\mbf{x} \in X} \snorm{\proj_W \mbf{x}}_2 \gtrsim 1/\sqrt{\log(2d/k)}$.

Let $\mbf{y} = (y_1,\ldots,y_d)$ a uniform random vector in $\mb{S}(W)$. Let $\sigma_i = (\mb{E}y_i^2)^{1/2}$. We have
\begin{equation}\label{eq:sigma-sum}
	\sigma_1^2+\cdots+\sigma_d^2 = \mb{E}[y_1^2 + \cdots + y_d^2] = 1 
\end{equation}
and 
\begin{equation} \label{eq:sigma-max}
\sigma_i^2 = \frac{1}{k} \snorm{\proj_W(\mbf{e}_i)}^2 \le \frac{1}{k}.	
\end{equation}

Without loss of generality, assume that $1/\sqrt{k}\ge\sigma_1\ge\cdots\ge\sigma_d\ge 0$, so that $\sigma_i \le 1/\sqrt{i}$ for each $i$.
We claim that
\[
a_i \ge \sqrt{\frac{2}{3}} \sigma_i \qquad \text{ for all }1 \le i \le d-k/2.
\]
Indeed, for $i\le k$, we have $a_i\ge 1/\sqrt{3k/2} \ge \sqrt{3/2} \sigma_i$.
For $k < i\le d-\lfloor k/2 \rfloor$, we have $a_i = 1/\sqrt{\lfloor k/2 \rfloor + i} \ge   \sigma_i\sqrt{i/(\lfloor k/2 \rfloor + i)} \ge \sqrt{2/3} \sigma_i$.

We have $\mb{E}|y_i| \gtrsim (\mb{E}y_i^2)^{1/2} = \sigma_i$ since $y_i$ is distributed as the first coordinate of a random point on $\sigma_i \sqrt{k} \cdot \mb{S}(\mb{R}^k)$.

Putting everything together, we have
\begin{align*}
\snorm{\mbf{a}}_2 \, \sup_{\mbf{x} \in X} \snorm{\proj_W \mbf{x}}_2
&\ge 
\sup_{g \in G} \snorm{\proj_W g \mbf{a}}
\ge 
\mb{E} \sup_{g \in G} \sang{\mbf{a}, g \mbf{y}} 
\\
&\ge 
\mb{E}\sum_{1\le i\le d}a_i|y_i|
\gtrsim \sum_{i=1}^da_i\sigma_i
\gtrsim \sum_{i=1}^{d-k/2} \sigma_i^2 
\ge \frac{1}{2},
\end{align*}
where the final step uses \cref{eq:sigma-sum,eq:sigma-max}. Thus 
\[
\sup_{\mbf{x} \in X} \snorm{\proj_W \mbf{x}}_2 \gtrsim \frac{1}{\snorm{\mbf{a}}_2} \gtrsim 1/\sqrt{\log(2d/k)}. \qedhere 
\]
\end{proof}

\bibliographystyle{amsplain0.bst}
\bibliography{main.bib}

\end{document}